\documentclass[DIV=classic,a4paper,10pt]{myart}

\KOMAoptions{DIV=last}

\usepackage{caption}
\usepackage{subcaption}

\newcommand{\norm}[1]{\left\Vert#1\right\Vert}
\newcommand{\abs}[1]{\left\vert#1\right\vert}

\renewcommand{\mid}{\,|\,}

\usepackage{accents}

\newcommand{\diffns}{\mathrm{d}}
\usepackage[dvipsnames]{xcolor}

\begin{document}

\title{Strong solutions of forward-backward stochastic differential equations with measurable coefficients}

\author[]{Peng Luo}
\author[]{Olivier Menoukeu-Pamen}
\author[]{Ludovic Tangpi}

\abstract{
This paper investigates solvability of fully coupled systems of forward-backward stochastic differential equations (FBSDEs) with irregular coefficients.
In particular, we assume that the coefficients of the FBSDEs are merely measurable and bounded in the forward process.
We crucially use compactness results from the theory of Malliavin calculus to construct strong solutions.
Despite the irregularity of the coefficients, the solutions turn out to be differentiable, at least in the Malliavin sense and, as functions of the initial variable, in the Sobolev sense.
}

\date{\today}
\keyAMSClassification{60H10, 34F05, 60H07, 35D40}
\keyWords{Singular PDEs, Sobolev regularity, FBSDE, singular coefficients, strong solutions, Malliavin calculus.}
\maketitle


\section{Introduction}

The main result of this work concerns the existence of a (strong) solution of the forward-backward stochastic differential equation (FBSDE)
\begin{equation}
\label{eq:fbsde}
	\begin{cases}
		X_t = x + \int_0^tb(u,X_u, Y_u, Z_u)\,\diffns u + \int_0^t\sigma \,dW_{u}\\
		Y_t = h(X_T) + \int_t^Tg(u,X_u,Y_u, Z_u)\diffns u - \int_t^TZ_u\,dW_{u}
	\end{cases}
\end{equation}
with $b,g$ and $h$ measurable in $(t,x)$, and uniformly continuous in $(y,z)$, see Theorem \ref{thm:exists fbsde}.
The proof of this result is partly inspired from results by \citet{Ma-Zhang11} and \citet{DeGua06} on \emph{weak} solutions of FBSDE under similar conditions.
Our contribution in this direction is to obtain \emph{strong} solutions and allow irregularity of $h$.
Because of the lack of regularity of the coefficients, usual fixed point and Picard iterations techniques cannot be applied here.
Let us briefly describe our method:

We start as in \cite{Ma-Zhang11,DeGua06} by approximating the functions $b,g$ and $h$ by smooth function, e.g. by mollification.
The FBSDE associated to these functions admit unique solutions $(X^n,Y^n,Z^n)$ and a so-called decoupling field $v_n$ which is the classical solution of an associated quasilinear PDE.
The function $v_n$ is called a decoupling field because it holds \begin{equation}
\label{eq:decoupling intro}
	Y_t^n = v_n(t,X^n_t)\quad \text{and}\quad Z^n_T = D_xv_n(t,X^n_t)\sigma,
\end{equation}
which allows to decouple the system.
The problem is now to derive strong limits for the above sequences and to show that these limits satisfy the desired equation.
Using classical a priori estimations for such equations, (see e.g. \cite{Ladyz-Sol-Ura68} or the statements recalled in the Appendix) it can be shown that for for every $\delta>0$ and every $t \in [0,T-\delta]$ the sequence of functions $v_n$ admits some compactness properties allowing to derive a limit $v$ for $v_n$ and a limit $w$ for $D_xv_n$.
When $h$ is sufficiently regular, say H\"older continuous, $\delta$ can be taken equal to zero.
In this setting, the idea of \cite{Ma-Zhang11,DeGua06} is to also gain sufficiently good control over the time-derivative and the Hessian using e.g. Calderon-Zygmund theory.
The approach proposed here is to rather use ideas from Malliavin calculus, notably the compactness principle due to \citet{DPMN92}, to find a limit $X$ of the sequence $(X^n)$ in the strong sense.
Together with the representation \eqref{eq:decoupling intro}, this allows to find strong limits for $Y$ and $Z$ (at least for $t$ small enough).
It remains to verify that the limiting processes $(X, Y, Z)$ actually solve the desired equation.


We further study regularity properties of solutions.
In fact, despite the singularity of the coefficients, it turns out that the solutions enjoy satisfactory regularity, at least in the Malliavin and Sobolev sense.
These are interesting results in that, the convention in the field is that solutions inherit the regularity properties of the coefficients \cite{Luo-Tangpi17,Ank-Imk-Reis07}.

FBSDEs are an essential tool in the investigation of stochastic control problems and stochastic differential games.
Due to Pontryagin's stochastic maximum principle, they can be used to characterize optimal controls and Nash equilibriums \cite{Pen90,MR3752669}.
These equations also provide a probabilistic approach to deal with quasilinear parabolic partial differential questions via the nonlinear Feynman-Kac formula initiated by \citet{Pardoux-Peng92} and further developed notably in \cite{kobylanski01,Bah-Edd-Ouk17,Delarue03,Par-Tang99}.
As a result, FBSDEs have received a lot of attention in the applied probability community and appear in various applications, we refer for instance to \cite{Hor-etal,optimierung,Cvi-Zha,Fro-Imk-Pro,Mik-Thi06} and the references therein.
When the coefficients of the equations, i.e. the functions $b, g $ and $h$ are sufficiently smooth, solvability of \eqref{eq:fbsde} is well-understood.
Refer for instance to \cite{MPY94,Delarue,Peng-Wu99} for the case of equations with Lipschitz continuous coefficients and to \cite{Kup-Luo-Tang18,Luo-Tangpi17} for locally Lipschitz coefficients.
When the coefficients are not regular enough, while an SDEs theory is well-developed (see e.g. \cite{MMNPZ13,SDERough,MNP2015,Bahlali99,KR05}) BSDEs with irregular coefficients are less well-studied.
A notable exception is the notion of \emph{weak solution} of FBSDE (very analogous to weak solutions of SDEs) introduced by \citet{Buck-Eng05} and further investigated in \cite{DeGua06,Ma-Zhang11,Ma-Zhang-Zheng}.
These solutions are constructed on a probability space that is possibly different from the underlying probability space. 
On the other hand, more recently, \citet{Isso-Jing20} studied two new classes of multidimensional FBSDEs with distributional coefficients.
In many applications, for instance to the construction of feedback solutions of stochastic control problems, it is important to have \emph{strong solutions}, and to analyze regularity properties thereof.



The remainder of the paper is organized as follows:
In the next section, we make precise the mathematical setting of the work and state the main results: Existence of strong solutions for FBSDEs with rough coefficients.
The proof is given in Section \ref{sec:fbsde proof}.
The regularity of the solutions of the FBSDE is analyzed in Section \ref{sec:regu fbsde}.
We consider both regularity in the Malliavin (variational) sense and in the Sobolev sense.

\section{Setting and main results}
Let $T \in (0,\infty)$ and $d \in \mathbb{N}$ be fixed and consider a probability space $(\Omega, {\cal F}, P)$ equipped with the completed filtration $({\cal F}_t)_{t\in[0,T]}$ of a $d$-dimensional Brownian motion $W$.
Throughout the paper, the product $\Omega \times [0,T]$ is endowed with the predictable $\sigma$-algebra. Subsets of $\mathbb{R}^k$, $k\in\mathbb{N}$, are always endowed with the Borel $\sigma$-algebra induced by the Euclidean norm $|\cdot|$.
Let us consider the following conditions:
\begin{enumerate}[label = (\textsc{A1}), leftmargin = 30pt]
	\item The function $b:[0,T]\times\mathbb{R}^d\times \mathbb{R}^l\times \mathbb{R}^{l\times d}\to \mathbb{R}^d$ is Borel measurable and it
	 holds
	\begin{equation*}
		|b(t, x,y,z)| \le k_1( 1+|y| +|z|)
	\end{equation*}	
	for some $k_1\ge 0$ and every $(x,y,z) \in \mathbb{R}^d\times\mathbb{R}^l\times\mathbb{R}^{l\times d}$.
	Moreover, for each fixed $(t,x)$ the restriction of $b(t,x,\cdot,\cdot)$ to the ball $$B_{R}(0) := \{(y,z): |y|\le R; \,\, |z|\le R \}$$ is continuous, with $R :=k_3e^{Tk_2}.$
	\label{a1}
\end{enumerate}
\begin{enumerate}[label = (\textsc{A2}), leftmargin = 30pt]
	\item $\sigma \in \mathbb{R}^{d\times d}$ and $\xi\sigma\sigma^*\xi>\Lambda|\xi|^2$ for some $\Lambda>0$ and for all $\xi \in \mathbb{R}^d$.\label{a2}
\end{enumerate}
\begin{enumerate}[label = (\textsc{A3}), leftmargin = 30pt]
	\item The function $g:[0,T]\times \mathbb{R}^d\times \mathbb{R}^l\times \mathbb{R}^{l\times d}\to \mathbb{R}^l$ is measurable, uniformly continuous in $(y,z)$, uniformly in $(t,x)\in [0,T]\times \mathbb{R}^d$ and satisfies
	\begin{equation*}
		 |g(t,x,y,z)|\le k_2(1 + |y| + |z|)
	\end{equation*}
	for some $k_2\ge0$,
	 and for every $(t,x,y,z)\in[0,T]\times \mathbb{R}^d\times\mathbb{R}^{l}\times \mathbb{R}^{l\times d}$.
	 \label{a3}
\end{enumerate}
\begin{enumerate}[label = (\textsc{A4}), leftmargin = 30pt]
	\item The function $h:\mathbb{R}^d\to \mathbb{R}^l$ is measurable and satisfies
	\begin{equation*}
	 	|h(x)|\le k_3
	 \end{equation*}
	 for some $k_3\ge0$ and for every $x \in \mathbb{R}^d$.\label{a4}
\end{enumerate}
The following is our first main result: In its statement, the space ${\cal S}^2(\mathbb{R}^d) \times{\cal S}^2(\mathbb{R}^{ l})\times{\cal H}^2(\mathbb{R}^{l\times d})$ is defined as follows:
For $p \in [1, \infty]$ and $k\in\mathbb{N}$, denote by ${\cal S}^p(\mathbb{R}^k)$ the space of all adapted continuous processes $X$ with values in $\mathbb{R}^k$ such that $\norm{X}_{{\cal S}^p(\mathbb{R}^k)}^p :=  E[(\sup\nolimits_{t \in [0,T]}\abs{X_t})^p] < \infty$, and by ${\cal H}^p(\mathbb{R}^{k})$ the space of all predictable processes $Z$ with values in $\mathbb{R}^{k}$ such that $\norm{Z}_{{\cal H}^p(\mathbb{R}^{k})}^p := E[(\int_0^T\abs{Z_u}^2\diffns u)^{p/2}] < \infty$.

\begin{theorem}
\label{thm:exists fbsde}
	Assume that the conditions \ref{a1}-\ref{a4} hold and that one of the following assumptions is satisfied:
	\begin{enumerate}[label = (\textsc{B1}), leftmargin = 30pt]
		\item $b$ and $g$ are bounded in $z$, i.e. $|b(t,x,y,z)|+ |g(t,x,y,z)|\le C(1 + |y|)$ for all $t,x,y,z$ for some $C\ge0$.
		\label{b1}
	\end{enumerate}
	\begin{enumerate}[label = (\textsc{B2}), leftmargin = 30pt]
		\item $h$ is Lipschitz continuous: $|h(x) - h(x')| \le k_3|x - x'|$ for every $x,x' \in \mathbb{R}^d$.
		\label{b2}
	\end{enumerate}
	Then the FBSDE \eqref{eq:fbsde} admits a solution $(X, Y ,Z) \in {\cal S}^2(\mathbb{R}^d) \times{\cal S}^2(\mathbb{R}^{ l})\times{\cal H}^2(\mathbb{R}^{l\times d})$ such that
	\begin{equation*}
		Y_t = v(t, X_t),\quad Z_t = w(t, X_t)\sigma\quad P\otimes dt\text{-a.s.}
	\end{equation*}
	for some measurable functions $v:[0,T]\times \mathbb{R}^d\to \mathbb{R}^l$ and $w:[0,T]\times \mathbb{R}^d \to \mathbb{R}^{l\times d}$.
\end{theorem}

The proof of Theorem \ref{thm:exists fbsde} will be given in Subsection \ref{sec:fbsde proof}.

\section{FBSDEs with measurable coefficients}
\subsection{proof of Theorem \ref{thm:exists fbsde}}
\label{sec:fbsde proof}
This section is entirely dedicated to the proof of Theorem \ref{thm:exists fbsde}.
Throughout, the conditions \ref{a1}-\ref{a4} are in force.
Let $(b_n)_n$, $(g_n)_n$ and $(h_n)_n$ be sequences of smooth functions with compact support converging pointwise to $b$, $g$ and $h$, respectively (e.g. obtained by standard mollification).
We can assume without loss of generality that for each $n$, the functions $h_n$ and $g_n$ satisfy \ref{a3}-\ref{a4}
in addition to being smooth and Lipschitz continuous (but with Lipschitz constant possibly depending on $n$).
These sequences will be used throughout the proof.
We begin the proof with the following simple lemma which shows that the sequence $b_n$ can be chosen so that the convergence holds uniformly on a given compact in $(y,z)$ and $g_n$ such that the convergence holds locally uniformly in $(y,z)$.
This will be needed at the end of the proof.
\begin{lemma}
\label{lem:molli conver}
	The sequence of mollifiers $(g_n)$ converges to $g$ pointwise in $(t,x)$ and uniformly in $(x,y)$.
	That is, for every $t,x$ it holds that
	\begin{equation*}
		\lim_{n\to \infty}\sup_{(y,z) \in \mathbb{R}^{l+ l\times d}}|g_n(t,x,y,z) - g(t,x,y,z)| = 0.
	\end{equation*}
	Similarly, $(b_n)$ converges to $b$ pointwise in $(t,x)$ and uniformly in $(y,z)$ on the ball of radius $R$ centered at the origin.
\end{lemma}
\begin{proof}
	Let $(\phi_n)$ be a sequence of standard mollifiers such that for each $n$, the support of $\phi_n$ is in the closure of the ball  $B_{1/n}(0)=\{ (y,z): |(y,z)|\le 1/n \}$.
	Let
	$\varepsilon>0$ be fixed.
	Since $g(t,x,\cdot,\cdot)$ is uniformly continuous,
	there is $\eta>0$ such that for $(y,z),(y',z')\in \mathbb{R}^{l+l\times d}$, satisfying  $|y -y'|+|z-z'|\le \eta$, it holds
	$|g(t,x,y,z) - g(t,x,y',z')|<\varepsilon$.
	Let $n \in \mathbb{N}$ and denote $\beta:= (y,z)$.
	Then, it holds that
	\begin{align*}
		\sup_{(y,z)}|g_n(t,x,y,z) - g(t,x,y,z) |&= \sup_{\beta = (y,z)}| \int_{[0,T]\times\mathbb{R}^{d+ l+ l\times d}}g((t,x,\beta) - \alpha)\phi_n(\alpha)\,d\alpha - g(t,x,\beta) |\\
								&\le \sup_{\beta =(y,z) } \int_{[0,T]\times B_{1/n}(0)}| g((t,x,\beta) - \alpha) - g(t,x,\beta) |\phi_n(\alpha)\,d\alpha\\
								&<\varepsilon\int_{[0,T]\times B_{1/n}(0)}\phi_n(\alpha)\,d\alpha = \varepsilon.
	\end{align*}
	This yields the result.

	The proof the local uniform convergence of $b_n$ is the same.
\end{proof}

\paragraph{Step1: Construction of an approximating sequence of solutions.}
Let $n\in \mathbb{N}$ be fixed.
According to \cite[Theorem 2.6]{Delarue}, for every $(s,x)\in [0,T]\times\mathbb{R}^d$ the FBSDE
\begin{equation}
\label{eq:fbsde n}
 	\begin{cases}
		X_t = x + \int_s^tb_n(u,X_u, Y_u, Z_u)\,\diffns u + \int_s^t\sigma \,dW_{u}\\
		Y_t = h_n(X_T) + \int_t^Tg_n(u,X_u,Y_u, Z_u)\diffns u - \int_t^TZ_u\,dW_{u}\quad t\in [s,T]
	\end{cases}
\end{equation}
admits a unique solution $(X^{s, x,n}, Y^{s, x,n}, Z^{s, x,n})\in {\cal S}^2(\mathbb{R}^d)\times {\cal S}^\infty(\mathbb{R}^l)\times {\cal H}^2(\mathbb{R}^{l\times d})$.
Denote by ${\cal L}^n$ the differential operator
\begin{equation*}
  	{\cal L}^nv := b_n(t,x, v, D_x v\sigma)D_x v + \frac 12\text{trace}(\sigma\sigma^*D_{xx} v),
\end{equation*}
  By   \cite[Theorem VII.7.1]{Ladyz-Sol-Ura68} (or see also \cite[Proposition 3.3]{MPY94}) the PDE
\begin{equation}
\label{eq:pde n}
	\begin{cases}
		\partial_tv_n(t,x) + {\cal L}^nv_n(t,x) + g_n(t,x, v_n(t,x),D_x v_n(t,x)\sigma) = 0\\
		v_n(T,x) = h(x)
	\end{cases}
\end{equation}
 admits a unique (classical) solution $v_n \in C^{1,2}([0,T]\times \mathbb{R}^d)$ that is bounded and with bounded gradient.
 Moreover, the solutions of \eqref{eq:pde n} and \eqref{eq:fbsde n} are linked through the identities (see \cite{MPY94})
\begin{equation}
\label{eq:link yn xn}
	Y^{s,x,n}_t = v_n(t, X^{s,x,n}_t) \quad \text{and}\quad Z^{s,x,n}_t = D_x v_n(t,X^{s,x,n}_t)\sigma, \quad t\in [s,T].
\end{equation}

The rest of the proof will consist in proving (strong) convergence of the above defined sequence of stochastic processes $(X^{s, x,n}, Y^{s, x,n}, Z^{s, x,n})$ and to verify that the limiting process satisfies the FBSDE with measurable drift.
Our method will make use of a priori (gradient) estimates for Sobolev solutions of parabolic quasilinear PDEs which can be found e.g. in \cite{Delarue03} or \cite{Ladyz-Sol-Ura68} and that we recall in the Appendix.
These estimates allow us to have:
\begin{lemma}
\label{lem:a priori pde}
	There are constants $C$ and $\alpha' \in (0,1)$ depending on $k_1, k_2, k_3,\sigma, d,l$ and $T$, and which do not depend on $n$ such that
	\begin{equation*}
		|v_n(t,x)|\le R \quad \text{for all}\quad (t,x) \in [0,T]\times \mathbb{R}^d
	\end{equation*}
	and for every $\delta>0$, there is a constant $C_\delta$ such that
	\begin{equation}
	\label{eq:Lipschitz vn}
	 	|D_xv_n(t,x)| \le C_\delta\quad \text{for every}\quad (t,x) \in [0,T-\delta]\times \mathbb{R}.
	 \end{equation}
	
	Moreover, if $h$ is $\alpha$-H\"older continuous, then
	\begin{equation}
	\label{eq:holder vn}
 		|v_n(t,x) - v_n(t',x')|\le C(|t-t'|^{\alpha'/2} + |x-x'|^{\alpha'})
	\end{equation}
	for every $(t,x), (t',x') \in [0,T]\times \mathbb{R}^d$ and some $\alpha'\in (0,\alpha]$.
 	If $h$ is Lipschitz continuous, then \eqref{eq:Lipschitz vn} holds with $\delta = 0$.
\end{lemma}
\begin{proof}
	The boundedness of $v_n$ is well-known.
	We provide it to explicitly derive the constant $R$.
	We have
	\begin{align*}
		v_n(t,x) = Y^{t,x,n}_t& = h_n(X^{t,x,n}_T) + \int_t^T\int_0^1\partial_zg_n(u, X^{t,x,n}_u, Y^{t,x,n}_u, \lambda Z^{t,x,n}_u)\,d\lambda Z^{t,x,n}_u\,du - \int_t^T Z^{t,x,n}_u\,dW_u\\
		&\quad +\int_t^T g_n(u, X^{t,x,n}_u, Y^{t,x,n}_u,0)\,du.
	\end{align*}
	Therefore, by the Girsanov's theorem, conditions \ref{a3}-\ref{a4} and Gronwall's inequality we have
	\begin{equation*}
		|v_n(t,x)|\le k_3e^{Tk_2} =R \quad \text{for all}\quad (t,x)\in [0,T]\times \mathbb{R}^d
	\end{equation*}
	and \eqref{eq:holder vn} follows by Theorem \ref{thm:pde estimate}.
	Furthermore, since $v_n$ is a classical solution of \eqref{eq:pde n}, i.e. $v_n \in C^{1,2}([0,T]\times \mathbb{R}^d)$, it is in particular a Sobolev solution, and $v_n \in W^{1,2}_{d+1,\mathrm{loc}}([0,T]\times \mathbb{R}^d, \mathbb{R}^l)$ (see definition in Appendix).
	Moreover, if $h$ is Lipschitz continuous then by definition of $(h_n)$, it holds $|h_n(x) - h_n(x')| \le k_3|x-x'|$  for every $x, x' \in \mathbb{R}^d$ and all $n\in \mathbb{N}$.
	Therefore, the last claims follow by Theorem \ref{thm:pde estimate Delarue}.
\end{proof}

\paragraph{Step 2: Candidate solution for the forward equation.}
In this step, we show that the sequence $(X^{s,x,n})_n$ converges in the strong topology of ${\cal S}^2(\mathbb{R}^d)$.
We first show existence of a weak limit.
To ease the presentation, we omit the superscript $(s,x)$ and put
\begin{equation*}
	X^n:= X^{s,x,n}, \quad Y^n:= Y^{s,x,n}\quad \text{and}\quad Z^n:= Z^{s,x,n}.
\end{equation*}

\paragraph{Step 2a: Weak limit.}
It follows from Step 1 that the process $X^n$ satisfies the forward SDE
\begin{equation}
\label{eq:SDE b tilde}
	X^n_t = x + \int_s^tb_n(u, X^n_u, v_n(u, X^n_u), D_x v_n(u, X^n_u)\sigma)\,du + \int_s^t\sigma dW_u.
\end{equation}
\begin{lemma}
\label{lem:linear tilde b}
	Consider the function $\tilde b_n:(t,x)\mapsto b_n(t, x, v_n(t, x), D_x v_n(u, X^n_u)\sigma)$.
	Under either of the conditions \ref{b1} or \ref{b2}, the function $\tilde b_n$  is continuously differentiable and uniformly bounded, i.e. there is a constant $C\ge0$ which does not depend on $n$ such that
	\begin{equation*}
		|\tilde b_n(t,x)|\le C \quad	\text{for all}\quad (t,x) \in [0,T]\times\mathbb{R}^d.
	\end{equation*}
\end{lemma}
\begin{proof}
	That $\tilde b_n$ is continuously differentiable follows from the fact that $b_n$ is smooth and $v_n$ is twice continuously differentiable.
	By \ref{a1} and Lemma \ref{lem:a priori pde}, if condition \ref{b1} holds, then for every $(t,x)\in [0,T]\times \mathbb{R}^d$ we have
	\begin{align*}
		|\tilde b_n(t,x)|& \le k_1(1 + |v_n(t,x)|)\\
		   				 & \le k_1(1 + C ).
	\end{align*}
	When condition \ref{b2} holds, it follows by Lemma \ref{lem:a priori pde} that $D_x v_n$ is bounded.
	Thus the result follows from the linear growth of $b$, i.e. \ref{a1}.
\end{proof}
Due to Lemma \ref{lem:linear tilde b}, it follows from standard SDE estimates that the sequence $(X^n)$ satisfies
\begin{equation*}
\label{eq:S2 bound}
	\sup_nE\left[\sup_{t\in [s,T]}|X^n_t|^2 \right]<\infty.
\end{equation*}
Therefore $(X^n)_n$ admits a subsequence which converges weakly to some process $\tilde X\in {\cal S}^2(\mathbb{R}^d)$.
This subsequence will be denoted again $(X^n)$.

\paragraph{Step 2b: Strong limit.}
Since $\tilde b_n$ is Lipschitz continuous, the solution $X^n$ of the SDE \eqref{eq:SDE b tilde} is Malliavin differentiable and since $\tilde b_n$ is a smooth function with compact support,
 it follows by \cite[Lemma 3.5]{MMNPZ13} that
\begin{equation}
\label{eq:comp crit}
	E\Big[\Big\|D_{t'}^i X^{n}_r -  D_t^i X^{n}_r \Big\|^2\Big] \le {\cal C}_{d,T}( \|\tilde b_n\|_\infty)|t-t'|^\alpha
\end{equation}
and
\begin{equation}
\label{eq: bound Mall derv X}
	\sup_{0\le t\le T}E\Big[\Big\| D_tX^{n}_r \Big\|^2\Big] \le {\cal C}_{d,T}( \|\tilde b_n\|_\infty)
\end{equation}
for a strictly positive constant ${\cal C}_{d,T}(\| \tilde b_n\|_\infty)$ such that ${\cal C}_{d,T}$ is a continuous increasing function, and with $\alpha = \alpha(r)>0$.
Since the sequence $\tilde b_n$ is bounded (see Lemma \ref{eq:SDE b tilde}), it follows that the bounds on the right hand sides of \eqref{eq:comp crit} and \eqref{eq: bound Mall derv X} do not depend on $n$. 

Therefore, it follows from the relative compactness criteria from Malliavin calculus of \cite{DPMN92} that the sequence $(X^n_r)_n$ admits a subsequence $(X^{n_k}_r)_k$ converging to some $X_r$ in $L^2$.

It remains to show that the choice of the subsequence $(X^{n_k}_r)_k$ does not depend on $r$.
That is, for every $t \in [s,T]$, $(X^{n_k}_t)_k$ converges to $X_t$ in $L^2$.
In fact, we will show that the whole sequence converges.
This is done as in the proof of \cite[Proposition 2.6]{SDERough}.
Assume by contradiction that for some $t\in [s,T]$, there is a subsequence $(n_k)_{k\ge0}$ such that
\begin{equation}
\label{eq:contradiction ineq}
	\| X^{n_k}_t - X_t\|_{L^2} \ge \varepsilon.
\end{equation}
Since \eqref{eq:comp crit} is proved for arbitrary $n$, it follows again by the compactness criteria of \cite{DPMN92} that $(X^{n_k})_k$ admits a further subsequence $(X^{n_{k_1}}_t)_{k_1}$ which converges in $L^2$ to $ X_t$.
But since we showed in Step 2a that the whole sequence of process $(X^n)_n$ converges weakly to the process $\tilde X$, it follows that $(X^{n_{k_1}}_t)_{k_1}$ converges weakly to $\tilde X_t$ and therefore, by uniqueness of the limit, $\tilde X_t = X_t$.
Since by \eqref{eq:contradiction ineq} it holds
\begin{equation*}
	\| X^{n_{k_1}}_t - X_t\|_{L^2} \ge \varepsilon,
\end{equation*}
we have a contradiction.
Thus,
\begin{equation*}
	X^n_t \to X_t \quad \text{in}\quad L^2 \quad\text{for every } t\in [s,T].
\end{equation*}
\paragraph{Step 3: Candidate solution for the value process $Y$ and the control process $Z$.}
	In this part we show that the sequence $(Y^n, Z^n)$ converges strongly in ${\cal H}^2(\mathbb{R}^l)\times {\cal H}^2(\mathbb{R}^{l\times d})$ to some $(Y, Z)$.

	First recall that $(Y^n)$ is a bounded sequence.
	Thus, it admits a subsequence again denoted $(Y^n)$ which converges weakly in $\mathcal{H}^2(\mathbb{R}^l)$ to some $Y$.
	We will show that the convergence is actually strong, provided that we restrict ourselves to a small enough time interval.
	Let $\delta\in (0,T)$ be fixed.
	By Lemma \ref{lem:a priori pde}, the sequence of functions $(v_n)$ is bounded and equicontinuous on $[0,T - \delta]\times \mathbb{R}^d$.
	Thus, by the Arzela-Ascoli theorem, there is a subsequence again denoted $(v_n)$ which converges locally uniformly to a continuous function $v^\delta$.
	Since by Lemma \ref{lem:a priori pde} the functions $v_n$ are H\"older continuous with a coefficient that does not depend on $n$ and with common H\"older exponents $\alpha'$ (in $x$) and $\alpha'/2$ (in $t$), we have
	\begin{align}
	\nonumber
		E\big[| v_n(t, X^n_t) - v^\delta(t, X_t)|^2\big] &\le E\big[| v_n(t, X^n_t) - v_n(t, X_t) |^2\big] + E\big[|v_n(t, X_t) - v^\delta(t,X_t)|^2\big]\\
		\label{eq:conv v}
		&\le CE\big[|X_t - X^n_t|^{2\alpha'}\big] + E\big[|v_n(t,X_t) - v(t,X_t)|^2\big]\to 0.
	\end{align}
	Therefore, $Y^n_t = v_n(t, X_t^n)$ converges to $v^\delta(t, X_t)$ in $L^2$ for each $t \in [0,T - \delta]$.
	It then follows by uniqueness of the limit that
	\begin{equation}
	\label{eq:y=v}
		Y_t := v^\delta(t,X_t) \quad \text{for all}\quad t\in [0,T - \delta].
	\end{equation}
	It then follows by Lebesgue dominated convergence (in view of Lemma \ref{lem:a priori pde}) that $(Y^n)$ converges to $Y$ in $\mathcal{H}^2(\mathbb{R}^l)$ restricted to $[0,T - \delta]$, i.e.
	\begin{equation}
	\label{eq:conv strong Y}
		\lim_{n\to \infty}E\Big[\int_0^{T - \delta}|Y^n_t - Y_t|^2\,dt \Big]=0.
	\end{equation}
	The equation \eqref{eq:y=v} further shows that $v^\delta$ does not depend on $\delta$.
	Thus, we will henceforth right
	$$
		Y_t = v(t,X_t)\quad \text{for all} \quad t\in [0,T- \delta]\quad \text{and}\quad \text{for all}\quad \delta>0.
	$$


	We now turn to the construction of the candidate control process $Z$.
	We want to justify that under both conditions \ref{b1} and \ref{b2} the sequence $b_n$ can be taken uniformly bounded.
	In fact, if the function $b$ satisfies \ref{b1}, and since $(Y^n)$ is uniformly bounded (this follows by the representation $v_n(t,X^n_t) = Y^{n}_t$ and Lemma \ref{lem:a priori pde}) it follows by uniqueness of solution that $(X^n, Y^n, Z^n)$ also solves the FBSDE \eqref{eq:fbsde n} with $b_n$ replaced by its restriction on $[0,T]\times \mathbb{R}^d\times B_R(0)\times \mathbb{R}^{l\times d}$.
	Similarly, if condition \ref{b2} holds, then $(Y^n)$ and $(Z^n)$ are bounded, and by uniqueness, $(X^n, Y^n, Z^n)$ also solves the FBSDE \eqref{eq:fbsde n} with $b_n$ replaced by its restriction on $[0,T]\times \mathbb{R}^d\times B_R(0)\times B_R(0)$.
	In particular, we can assume without loss of generality that $b_n$ is uniformly bounded, i.e. $|b_n(t,x,y,z)|\le C$ for all $n, t, x,y,z$ and for some constant $C>0$.
	Therefore, it follows by Theorem \ref{thm:pde estimate} that for every $\delta>0$ and $\kappa \in (0,1)$ there is a constant $C_{\delta,\kappa}$ independent on the derivatives of the coefficient (which in particular does not depend on $n$) such that for every $t,t' \in [0,T - \delta]$ and $x,x' \in \mathbb{R}^d$ it holds that
\begin{equation*}
	| D_x v_n(t,x) - D_x v_n(t',x') | \le C_{\delta,\kappa}(|x - x'|^{\kappa} + |t - t'|^{\kappa/2}).
\end{equation*}
Now, let $(\delta^k)$ be a strictly decreasing sequence converging to $0$.
By Arzela-Ascoli theorem, there is a subsequence $w_{n,k}:= {D_x v_n}_{|_{[0, T - \delta^k]\times \mathbb{R}^d}}$ which converges locally uniformly to some function $w_k$ on $[0,T - \delta^k]\times \mathbb{R}^k$.
Since $Z^n_t = D_x v_n(t, X^n_t)\sigma$ for all $t \in [0,T]$ (recall \eqref{eq:link yn xn}) we then have $Z^{n_k}_t = w_{n,k}(t, X^n_t)\sigma$ for every $t\in [0,T - \delta^k]$ and every $k \in \mathbb{N}$, for some subsequence of $Z^n$.
And arguing as in Equation \ref{eq:conv v}, we have
\begin{equation*}
	Z^{n_k}_t = D_x w_{n,k}(t, X^{n_k}_t)\sigma \to  w_k(t, X_t)\sigma \quad \text{in}\quad L^2\quad\text{for every}\quad t \in [0,T - \delta^k].
\end{equation*}
Assumption \ref{a2} and uniqueness of the limit show that $w_k = w_{k+1}$ on $[0,T - \delta^k]$ for every $k$.
Thus, the function
\begin{equation*}
	w(t,x) := w_1(t,x)1_{[0, T - \delta^1]}(t) + \sum_{k = 1}^\infty w_k(t,x)1_{[T - \delta^{k}, T - \delta^{k+1}]}(t)
\end{equation*}
is a well-defined Borel measurable function
and putting
\begin{equation}
\label{eq: z = w}
	Z_t := w(t, X_t)\sigma,
\end{equation}
we have by Lebesgue dominated convergence that $Z^{n_k}\to Z$ in $\mathcal{H}^2(\mathbb{R}^{l\times d})$ restricted to the interval $[0,T - \delta^k]$.
In particular, it follows by It\^o isometry that
\begin{equation}
\label{eq:conv stoc in}
	\int_0^{T - \delta^k}Z^{n_k}_t\,dW_t \to \int_0^{T - \delta^k}Z_t\,dW_t \quad \text{in}\quad L^2\quad \text{for every } k.
\end{equation}

\paragraph{Step 4: Verification.}
The goal of this step is to show that the triple of processes $(X, Y, Z)$ constructed above satisfies the coupled system \eqref{eq:fbsde}.
This part of the proof will be further split into 2 steps:
We first show that $(X,Y,Z)$ satisfies the forward equation.
This step uses the representations $Y_t= v(t,X_t)$ and $Z_t = w(t,X_t)\sigma$ in a crucial way.
In fact, these representation allow to obtain a solution $\bar X$ of a decoupled SDE with measurable drift that we can then show to coincide with the candidate solution $X$ constructed above.
In the last part we show that $(X,Y,Z)$ satisfies the backward equation.

\paragraph{Step 4a: The forward equation.}
Using either of the conditions \ref{b1} or \ref{b2}, we can show as above that the function  $x\mapsto b(t, x, v(t, x),w(t,x)\sigma)$ is bounded.
Therefore, \cite{MMNPZ13} gives existence of a unique solution $\bar X$ to the SDE
\begin{equation*}
	\bar X_t = x + \int_s^tb(u, \bar X_u, v(u, \bar X_u), w(u, \bar X_u)\sigma)\,du +\int_s^t\sigma\,dW_u.
\end{equation*}
Hence, in view of \eqref{eq:y=v} and \eqref{eq: z = w}, it remains to show that $\bar X_t = X_t$ $P$-a.s. for every $t\in [s,T]$ to conclude that the forward SDE is satisfied, that is, that
\begin{equation*}
	 X_t = x + \int_s^tb(u,  X_u, Y_u,Z_u)\,du +\int_s^t\sigma\,dW_u.
\end{equation*}
To that end, (by uniqueness of the limit) it suffices to show that for each $t \in [s,T]$ the sequence $(X^n_t)_n$ converges to $\bar X_t$ in the weak topology of $L^2(P)$.
Since the set
\begin{equation*}
 	\left\{ {\cal E}(\dot\varphi\cdot W)_{0,T}:\,\, \varphi \in C_b^1([0,T],\mathbb{R}^d) \right\}
 \end{equation*}
 is dense in $L^2(P)$, in order to get weak convergence it is enough to show that $(X^n_t{\cal E}(\dot\varphi_u\cdot W)_{0,T})_n$ converges to $X_t{\cal E}(\dot\varphi_u\cdot W)_{0,T}$ in expectation, for every $\varphi \in C_b^1([0,T],\mathbb{R}^d)$.
 Hereby $C_b^1([0,T],\mathbb{R}^d)$ denotes the space of bounded continuously differentiable functions on $[0,T]$ with values in $\mathbb{R}^d$, and $\dot\varphi$ is the derivative of $\varphi$.
Put $\tilde X^n_t(\omega):= X^n_t(\omega+ \varphi)$ and $\tilde X_t(\omega):=\bar X_t(\omega+ \varphi)$.
It follows by the Cameron-Martin theorem, see e.g. \cite{UsZa1} that $\tilde X^n$ satisfies the SDE
\begin{equation*}
	d\tilde X^n_t = \left(b_n(t, \tilde X^n_t, v_n(t, \tilde X^n_t), D_x v_n(t, \tilde X^n_t)\sigma) + \sigma \dot\varphi_t\right)\,dt +\sigma dW_t.
\end{equation*}
In fact, for every $H \in L^2(P; {\cal F}_t)$, it holds
\begin{align*}
	E\left[\tilde X^n_tH \right] & = E\left[X^n_tH(\omega - \varphi){\cal E}(\dot\varphi_u\cdot W)_{s,T} \right]\\
				&=E\left[\left(x + \int_s^tb_n(u, X^n_u, v_n(u, X^n_u),D_x v_n(u, X_u^n)\sigma)\,du + \sigma W_t \right)H(\omega-\varphi){\cal E}(\dot\varphi_u\cdot W)_{s,T} \right]\\
				&=E\left[\left(x + \int_s^tb_n(u, X^n_u, v_n(u, X^n_u),D_x v_n(u, X^n_u)\sigma)(\omega+ \varphi)\,du + \sigma W_t(\omega+\varphi) \right)H \right]\\
				&=E\left[\left(x + \int_s^tb_n(u, \tilde X^n_u, v_n(u, \tilde X^n_u),D_x v_n(u, \tilde X^n_u)\sigma) + \sigma \dot\varphi_u\,du + \sigma W_t(\omega) \right)H \right],
\end{align*}
where the latter equality follows by the fact that $W_t(\omega + \varphi) = W_t(\omega) +\varphi_t$, since $W$  is the canonical process.
This proves the claim.
That $\tilde X$ satisfies
\begin{equation*}
	d\tilde X_t = \left(b(t, \tilde X_t, v(t, \tilde X_t), w(t,\tilde X_t)\sigma) + \sigma \dot\varphi_t\right)\,dt +\sigma dW_t
\end{equation*}
is proved similarly.
Now put
\begin{equation*}
	u_n(t, x):= \sigma^*(\sigma\sigma^*)^{-1}b_n(t, x, v_n(t, x),D_x v_n(t,x)\sigma)\quad\text{and}\quad u:= \sigma^*(\sigma\sigma^*)^{-1}b(t, x, v(t, x),w(t,x)\sigma).
\end{equation*}
Recall that the law of $\tilde X^n_t$ under the probability measure $Q^n$ with density ${\cal E}( u_n(r, \tilde X^n_r) + \dot\varphi_r\cdot W)_{0,T}$ coincides with the law of $x+ \sigma W_t$ under $P$.
Similarly, the law of $\tilde X_t$ under the probability measure $Q$ with density ${\cal E}( u(r, \tilde X_r) + \dot\varphi_r\cdot W)_{0,T}$ coincides with the law of $x+ \sigma W_t$ under $P$.
Thus, it follows by Girsanov's theorem and the inequality $|e^a - e^b|\le |e^a + e^b||a-b|$ that
\begin{align}\label{eqweaklim3}
	&E\left[X^n_t{\cal E}(\dot\varphi_u\cdot W)_{0,T} \right] - E\left[\tilde X_t{\cal E}(\dot\varphi_u\cdot W)_{0,T} \right]\\
	&= E\left[(x + \sigma W_t)\left({\cal E}\left( \{u_n(r, x+ \sigma W_r )+\dot\varphi_r\}\cdot W \right)_{0,T}- {\cal E}(\{ u(r, x+ \sigma W_r)+\varphi_r\}\cdot W)_{0,T} \right) \right]\notag\\
 		  &\leq C E\left[|x+\sigma\cdot W_{t}|^2\right]^{\frac{1}{2}}\\
 		 &\quad\times E\left[\left(\mathcal{E}\left(\Big\{ u_n\left(r,x+\sigma W_r\right)+\dot\varphi_r\Big\}\cdot W\right)_{0,T}+\mathcal{E}\left(\Big\{ u\left(r,x+\sigma  W_r\right)_{0,T}+\dot\varphi_r\Big\}\cdot W\right)_{0,T}\right)^4\right]^{\frac{1}{4}}\notag\\
 		&\quad \times\left\{E\left[ \left(\int_0^T\right.\left( u_n\left(r,x+\sigma W_r\right)-  u\left(r,x+\sigma  W_r\right)\right)dW_r\right)^4\right]\notag\\
 		&\quad +E\left[\left(\left.\int_0^T\left\{\| u_n(r,x+\sigma W_r)+\dot\varphi_r\|^2	-\| u(r,x+\sigma W_r)+\dot\varphi_r\|^2\right\}\diffns r\right)^4\right] \right\}^{\frac{1}{4}}\notag\\
 		&=I_1\times I_{2,n}\times (I_{3,n}+I_{4,n})^{1/4}.
 	\end{align}
 	That $I_1$ is finite is clear, by properties of Brownian motion.
 	Since $b_n$ is bounded, so is $u_n$.
 	Thus, by boundedness of $\dot\varphi$, it holds that $\sup_nI_{2,n}$ is finite.

 	Now if we show that the sequence $(u_n)$ converges to $u$ pointwise, it would follow by Lebesgue's dominated convergence theorem, to get that $I_{3,n}$ and $I_{4,n}$ converge to $0$ as $n$ goes to infinity, hence concluding the proof.
 	In fact, there is $R>0$ such that $|v_n|\le R$ and for every $t \in [0,T)$, there is $R'$ such that\footnote{Under the condition \ref{b1} and when $t=T$, the sequence $(D_x v_n)$ might not be bounded and $(u_n)$ does not necessarily converge to $u$ but convergence for almost every $t$ is enough.} $|D_x v_n|\le R'$.
 	Thus, by  definition of $u_n$ and $u$, and $(t,x)\in [0,T)\times \mathbb{R}^d$ we have
 	\begin{align*}
 	 	|u_n(t,x) - u(t,x)| &\le C|b_n(t,x, v_n(t,x),D_x v_n(t,x)\sigma) - b(t,x, v(t,x), w(t,x)\sigma)|\\
 	 	& \le C|b_n(\cdot, v_n,D_x v_n\sigma) - b(\cdot, v_n,D_x v_n\sigma)|(t,x) + C|b(\cdot, v_n,D_x v_n\sigma) - b(\cdot, v, w\sigma)| (t,x) \\
 	 	   &\le C\sup_{y \in B_R(0), z\in B_{R'}(0) }|b_n(t,x,y,z ) - b(t,x, y,z)|\\
 	 	   &\quad + C|b(t,x, v_n(t,x),v_n(t,x)\sigma) - b(t,x, v(t,x),w(t,x)\sigma)|.
 	 \end{align*}
 	 The first term converges to zero since $b_n$ converges to $b$ locally uniformly (in $(y,z)$); and the second term converges to zero because $v_n$ and $D_x v_n\sigma$ converge to $v$ and $w\sigma$ respectively, and the function $b(t,x,\cdot, \cdot)$ is continuous on the ball $B_R(0)\times B_{R'}(0)$.

	\paragraph{Step 4b: The backward equation.}
	In this final step of the proof we show that the process $(X,Y,Z)$ satisfies the backward equation.
	The arguments is very similar to those of the Step 4a and also rely on the existence of the decoupling fields $v$ $w$ and Girsanov's transform.

	By Steps 2 and 3 we know that $(X^n_t)$ converges to $X_t$ in $L^2$, $(Y^{n_k},Z^{n_k})$ converges to $(Y,Z)$ in $\mathcal{H}^2(\mathbb{R}^l)\times \mathcal{H}^2(\mathbb{R}^{l\times d})$ (restricted to the interval $[0,T - \delta^k]$).
	Let $k$ be fixed and let $X^{n_k}$ be a subsequence corresponding to $(Y^{n_k},Z^{n_k})$.
	For every $n,k$ we have
	\begin{align}
	\label{eq:bsde for nk}
		Y_t^{n_k} =Y^{n_k}_{T - \delta^k} + \int_t^{T - \delta^k}g_{n_k}(u,X^{n_k}_u, Y^{n_k}_u, Z^{n_k}_u)\,du - \int_t^{T - \delta^k}Z^{n_k}_u\,dW_u.
	\end{align}
	Now, we would like to take first the limit in $n_k$ and then limit in $k$ on both sides.
By Step 3, the sequences of random variables $Y^{n_k}_t$, $Y^{n_k}_{T-\delta^k}$ and $\int_t^{T-\delta^k}Z^{n_k}_u\,dW_u$ respectively converge to $Y_t$, $Y_{T-\delta^k}$ and $\int_t^{T-\delta^k}Z_u\,dW_u$ in $L^2$.
Thus, it suffices to show that $\int_t^{T-\delta^k}g_{n_k}(X^{n_k}_u, Y^{n_k}_u, Z^{n_k}_u)\,du$ converges to $\int_t^{T-\delta^k}g(u, X_u, Y_u, Z_u)\,du$ in $L^2$.
To this end, define
\begin{equation*}
 	\tilde g_{n_k}(t,x):= g_{n_k}(t,x, v_{n_k}(t,x), D_xv_{n_k}(t, x)\sigma) \quad \text{and}\quad \tilde g(t,x):= g(t,x, v(t,x), w(t, x)\sigma).
 \end{equation*}
Observe that $\tilde g_{n_k}$ converges to $g$ pointwise, for every $t \in [0,T-\delta^k]$.
In fact,
\begin{align*}
	|\tilde g_n(t,x) - \tilde g(t,x)| &= | g_{n}(t,x, v_{n}(t,x), D_xv_{n}(t, x)\sigma) - g(t,x, v_{n}(t,x), D_xv_{n}(t, x)\sigma) |\\
	&\quad + | g(t,x, v_{n}(t,x), D_xv_{n}(t, x)\sigma) - g(t,x, v(t,x), w(t, x)\sigma) | \\
	&\le \sup_{y,z}|g_n(t,x,y,z) - g(t,x,y,z)| \\
	&\quad + | g(t,x, v_{n}(t,x), D_xv_{n}(t, x)\sigma) - g(t,x, v(t,x), w(t, x)\sigma) | \to 0
\end{align*}
where we used Lemma \ref{lem:molli conver} and continuity of $g$ in $(y,z)$ .
Recall the representations $Y^{n_k}_u = v_{n_k}(u, X^{n_k}_u)$, $Z^{n_k}_u = D_xv^{n_k}(u, X^{n_k}_u)\sigma$ and $Y_u= v(u, X_u)$, $Z_u = w(u, X_u)\sigma$.
For any $m \in \mathbb{N}$, we have
\begin{align*}
	&E\Big[\int_t^{T-\delta^k} |g_{n_k}(u,X^{n_k}_u, Y^{n_k}_u, Z^{n_k}_u) - g(u, X_u, Y_u, Z_u)|^2\,du\Big]  = E\Big[\int_t^{T-\delta^k}|\tilde g_{n_k}(u,X^{n_k}_u) - \tilde g(u,X_u)|^2\,du\Big] \\
	&\le E\Big[\int_t^{T-\delta^k}|\tilde g_{n_k}(u, X^{n_k}_u) - \tilde g(u, X^{n_k}_u)|^2 + |\tilde g(u, X^{n_k}_u) - \tilde g_{m}(u,X^{n_k}_u) |^2 + |\tilde g_{m}(u,X^{n_k}_u)  - \tilde g(u, X_u)|^2\,du \Big]\\
	&\le E\Bigg[{\cal E}\big(\tilde b_{n_k}(u,x + \sigma W_u )\cdot W \big)_{0,T}\Big\{\int_t^{T-\delta^k}|\tilde g_{n_k}(u, x + \sigma W_u) - \tilde g(u, x + \sigma W_u)|^2\\
	&\quad + |\tilde g(u, x + \sigma W_u) - \tilde g_{m}(u,x + \sigma W_u) |^2\,du\Big\}\Bigg] + E\Big[ \int_t^{T-\delta^k} |\tilde g_{m}(u,X^{n_k}_u)  - \tilde g(u, X_u)|^2\,du \Big],
\end{align*}
where the last inequality follows by Girsanov's theorem and where we used the notation
\begin{equation}
\label{eq:tilde b}
	\tilde b_{n_k}(t,x):= b(t, x, v_{n_k}(t,x) , D_xv_{n_k}(t, x)\sigma).
\end{equation}
Therefore, using H\"older's inequality the above estimation continues as
\begin{align*}
	&E\Big[\int_t^{T-\delta^k} |g_{n_k}(u,X^{n_k}_u, Y^{n_k}_u, Z^{n_k}_u) - g(u, X_u, Y_u, Z_u)|^2\,du\Big]\\
	&\le CE\Big[{\cal E}\big(\tilde b_{n_k}(u,x + \sigma W_u )\cdot W \big)_{0,T}^2\Big]^{1/2}E\Big[\int_t^{T-\delta^k}|\tilde g_{n_k}(u, x + \sigma W_u) - \tilde g(u, x + \sigma W_u)|^4\\
	&\quad + |\tilde g(u, x + \sigma W_u) - \tilde g_{m}(u,x + \sigma W_u) |^4\,du\Big]^{1/2} + E\Big[ \int_t^{T-\delta^k} |\tilde g_{m}(u,X^{n_k}_u)  - \tilde g(u, X_u)|^2\,du \Big].
\end{align*}
Since $\tilde b_{n_k}$ is bounded, the quantity $E\Big[{\cal E}\big(\tilde b_{n_k}(u,x + \sigma W_u )\cdot W \big)_{0,T}^2\Big]$ is bounded.
Thus, letting $m$ fixed and taking the limit as $n_k$ goes to infinity we obtain by Lebesgue dominated convergence that
\begin{align*}
	&\lim_{n_k\to\infty}E\Big[\int_t^{T-\delta^k} |g_{n_k}(u,X^{n_k}_u, Y^{n_k}_u, Z^{n_k}_u) - g(u, X_u, Y_u, Z_u)|^2\,du\Big]\\
	&\le CE\Big[\int_t^{T-\delta^k}|\tilde g(u, x + \sigma W_u) - \tilde g_{m}(u,x + \sigma W_u) |^4\,du \Big] + E\Big[ \int_t^{T-\delta^k} |\tilde g_{m}(u,X_u)  - \tilde g(u, X_u)|^2\,du \Big].
\end{align*}
Letting $m$ go to infinity it follows again by dominated convergence that the right hand side above goes to zero.
Thus
\begin{equation*}
	\int_t^{T-\delta^k}g_{n_k}(u,X^{n_k}_u, Y^{n_k}_u, Z^{n_k}_u) \to \int_t^{T-\delta^k}g(u, X_u, Y_u, Z_u)\,du \quad \text{in}\quad L^2.
\end{equation*}
Hence, $(X,Y,Z)$ satisfies
\begin{equation*}
	Y_t = Y_{T - \delta^k} +\int_t^{T - \delta^k}g(u,X_u, Y_u, Z_u)\,du - \int_t^{T - \delta^k}Z_u\,dW_u\quad P\text{-a.s. for every } k.
\end{equation*}
Next, we take the limit as $k$ goes to infinity.
Since $\delta^k\downarrow 0$, we only need to justify that $(Y_{T - \delta^k})$ converges to $Y_T$ $P$-a.s.
Indeed, since $(Y^n_T)$ converges to $Y_T$ in the weak topology of $L^2$, there exists a subsequence $(\tilde{Y}^n_T)$ in the asymptotic convex hull of $(Y^n_T)$ such that $(\tilde{Y}^n_T)$ converges to $Y_T$ in $L^2$. Moreover, $\tilde{Y}^n_T$ satisfies
\begin{equation*}
\tilde{Y}^n_T=\tilde{Y}^n_t-\int_t^TG^n_udu+\int_t^T\tilde{Z}^n_udW_u
\end{equation*}
where $(\tilde{Y}^n_t, G^n_u, \tilde{Z}^n_u)$ is the convex combination of $(Y^n_t, g_n(u,X^n_u,Y^n_u,Z^n_u), Z^n_u)$ corresponding to $\tilde{Y}^n_T$. If the condition \ref{b1} is satisfied, then $|g_n(u, X^n_u, Y^n_u, Z^n_u)|$ is dominated by $|Y^n_u|$ which is bounded, and if the condition \ref{b2} is satisfied, then $Z^n_u = D_x v^n(u, X^n_u)\sigma$ is bounded (by Lemma \ref{lem:a priori pde}), thus it follows by \ref{a3} that $|g_n(u, X^n_u, Y^n_u, Z^n_u)|$ is bounded. Hence, $G^n_u$ is bounded under both conditions. Therefore it follows by triangular inequality that for every $k,n \in \mathbb{N}$ it holds that
\begin{align*}
	\big| Y_{T - \delta^k} &- E[Y_T \mid \mathcal{F}_{T - \delta^k}]\big|\\
	&\le C\big( | Y_{T - \delta^k} - \tilde{Y}^n_{T - \delta^k}| + | E[\tilde{Y}^n_{T - \delta^k} - \tilde{Y}^n_T\mid \mathcal{F}_{T - \delta^k}]| + E[ |\tilde{Y}^n_T- Y_T| \mid \mathcal{F}_{T - \delta^k}]  \big)\\
	&\le C\Big( | Y_{T - \delta^k} - \tilde{Y}^n_{T - \delta^k}| +  E\Big[\int_{T-\delta^k}^T|G^n_u|\,du \mid \mathcal{F}_{T - \delta^k} \Big] + E[ |\tilde{Y}^n_T- Y_T| \mid \mathcal{F}_{T - \delta^k}] \Big)\\
&\le C(|Y_{T - \delta^k} - \tilde{Y}^n_{T - \delta^k}| + \delta^k + E[ |\tilde{Y}^n_T- Y_T| \mid \mathcal{F}_{T - \delta^k}]).
\end{align*}
for some constant $C>0$. Since $(\tilde{Y}^n_T)$ converges to $Y_T$ in $L^2$, $(Y^n_{T-\delta^k})$ converges to $Y_{T-\delta^k}$ in $L^2$ and $\tilde{Y}^n_{T-\delta^k}$ is the convex combination of $Y^n_{T-\delta^k}$, taking the limit first in $n$ and then in $k$ as they go to infinity shows that $| Y_{T - \delta^k} - E[Y_T \mid \mathcal{F}_{T - \delta^k}] |\to 0$ $P$-a.s.
On the other hand, in our filtration every martingale has a continuous version.
Thus, $E[Y_T\mid \mathcal{F}_{T - \delta^k}] \to Y_T$ $P$-a.s. as $k$ goes to infinity.
We can therefore conclude that $Y_{T - \delta^k} \to Y_T$ $P$-a.s. when $k$ goes to infinity, which yields
\begin{equation*}
	Y_t = Y_T + \int_t^{T }g(u,X_u, Y_u, Z_u)\,du - \int_t^{T}Z_u\,dW_u.
\end{equation*}
It finally remains to show that $Y_T = h(X_T)$.
Since $(Y^n_T)$ converges to $Y_T$ in the weak topology of $L^2$ (see the beginning of Step 3) if we show that $(Y^n_T)$ converges to $h(X_T)$ in $L^2$ then we can conclude that $Y_T = h(X_T)$.
If \ref{b2} holds, this is clear.
In case \ref{b1} holds, this is done using again a Girsanov change of measure and boundedness of $\tilde b_n$ (recall definition given in \eqref{eq:tilde b}).
In fact, for every $m \in \mathbb{N}$ it holds that
\begin{align*}
	&E[|h_n(X^n_T) - h(X_T)|^2]\\
	 &\le C\Big( E[|h_n(X^n_T) - h(X_T^n) |^2] + E[|h(X^n_T) - h_m(X^n_T) |^2] + E[|h_m(X^n_T) - h(X_T)|^2] \Big)\\
		&\le C\Big(E\Big[{\cal E}\big(\tilde b_n(t, x+\sigma W_t)\cdot W \big)_{0,T}\Big\{ |h(x+\sigma W_T) - h_m(x+\sigma W_T) |^2 + |h_m(x+\sigma W_T) - h(x+\sigma W_T)|^2\Big\} \Big]\\
		&\quad +  E[|h_m(X^n_T) - h(X_T)|^2]\Big)\\
		&\le C\Big(E\Big[{\cal E}\big(\tilde b_n(t, x+\sigma W_t)\cdot W \big)_{0,T}^2\Big]^{1/2}\\
		&\quad\times E\Big[ |h(x+\sigma W_T) - h_m(x+\sigma W_T) |^4 +|h_m(x+\sigma W_T) - h(x+\sigma W_T)|^4\Big]^{1/2} \\
		&\quad +  E[|h_m(X^n_T) - h(X_T)|^2]\Big).
\end{align*}
Since $\tilde b_n$ is bounded, the first term on the right hand side above is bounded.
Thus, fix $m$ then take the limit $n \to \infty$ and then the limit $m\to \infty$ to get by dominated convergence
\begin{equation*}
	E[|h_n(X^n_T) - h(X_T)|^2]\to 0.
\end{equation*}

 	This concludes the proof.
\hfill$\Box$

\subsection{Regularity of solutions}
\label{sec:regu fbsde}
In this section we investigate regularity properties of the solution $(X,Y,Z)$ of the FBSDE \eqref{eq:fbsde}.
We will consider two types of regularity properties.
We start by proving Malliavin differentiability of the solution.
This follows as a direct consequence of the method of proof of the existence result.
Then, we continue to consider smoothness of the solution as function of the initial position of the forward process.
We will show that for each $s\in [0,T]$ and $t\ge s$, the mapping $x\mapsto (X^{s,x}_t, Y^{s,x}_t)$ belongs to a weighted Sobolev space for almost every path.
The last result will be central for applications to PDEs.

\subsubsection{Malliavin differentiability}
 Let $0\le s\le t\le T$ and $x \in \mathbb{R}^d$.
 Let $(X, Y ,Z)$ be the solution of FBSDE \eqref{eq:fbsde} given by Theorem \ref{thm:exists fbsde}. The next result gives the Malliavin differentiability of $(X,Y,Z)$. We additionally consider the following conditions:
 \begin{enumerate}[label = (\textsc{A5}), leftmargin = 30pt]
	\item
	The function $g(t,x,y,z) = g(t,x,y)$ does not depend on $z$ and is Lipschitz continuous in $(x,y)$.
	 \label{a5}
\end{enumerate}
 \begin{enumerate}[label = (\textsc{A6}), leftmargin = 30pt]
	\item The function $g(t,x,y,z) = g(t,y,z)$ does not depend on $x$ and is continuously differentiable in $(y,z)$ and is Lipschitz continuous in $(y,z)$.
	 \label{a6}
\end{enumerate}
\begin{proposition}
\label{pro:malliavin flow}
	Assume that the conditions \ref{a1}-\ref{a4} are satisfied.
	\begin{itemize}
		\item[(i)] If \ref{b1} is satisfied, then $X_t$ is Malliavin differentiable for all $t\in [0,T]$	and for every $\delta >0$, $Y_t$ is Malliavin differentiable for all $t\in[0,T - \delta]$.
		\item[(ii)] If \ref{b2} is satisfied, then $(X_t,Y_t)$ is Malliavin differentiable for all $t\in[0,T]$.
		\item[(iii)] If \ref{b2} and either of the conditions \ref{a5} or \ref{a6} holds, then $(X_t, Y_t,Z_t)$ is Malliavin differentiable for all $t\in[0,T]$.
	\end{itemize}
\end{proposition}
\begin{proof}
 Consider the sequence $(X^n)$ constructed in the proof of Theorem \ref{thm:exists fbsde}.
 Recall that under both \ref{b1} and \ref{b2} we have
 \begin{equation*}
	X^n_t \to X_t \quad \text{in}\quad L^2 \quad\text{for every } t\in [0,T]
\end{equation*}
and  (see Equation \eqref{eq:comp crit} with  $t'=0$ therein) we have
\begin{align*}
	E\left[| D_tX^n_s|^2 \right] \le \sum_{i=1}^dE\left[\Big\|D_t^i X^{n}_s \Big\|^2 \right] \le d{\cal C}_{d,T}( \|\tilde b_n\|_\infty)t
\end{align*}
where $\tilde b_n$ is a uniformly bounded sequence.
Therefore, by \cite[Lemma 1.2.3]{Nua06} we conclude that $X_t$ is Malliavin differentiable for all $t\in [0,T]$.
In particular, $\sup_tE[|D_tX_s|^2]<\infty$.
To deduce the differentiability of $Y$, recall that for every $\delta >0$ and every $t \in [0,T - \delta]$ the function $x\mapsto v(t, x)$ is Lipschitz continuous. Thus, it follows by chain rule (see \cite[Proposition 1.2.4]{Nua06}) that $Y_t$ is Malliavin differentiable for all $t\in[0,T - \delta]$.

When condition \ref{b2} is satisfied, the function $x \mapsto v(t,x)$ is Lipschitz continuous for every $t\in [0,T]$.
Again by chain rule, $Y_t$ is Malliavin differentiable for all $t\in[0,T]$.
Thus, $(X_t,Y_t)$ is Malliavin differentiable.

If furthermore condition \ref{a5} holds, then in view of the identity
\begin{equation*}
\int_t^TZ_sdW_s=h(X_T)-Y_t+\int_t^Tg(s,X_s,Y_s)ds,
\end{equation*}
it follows from the chain rule and \cite[Lemma 2.3]{Pardoux-Peng92} that
$Z_t$ is Malliavin differentiable for all $t\in[0,T]$.
If we rather assume \ref{a6}, then since $X_t$ is Malliavin differentiable, the Malliavin differentiability of $(Y_t,Z_t)$ follows from the chain rule and \cite[Proposition 5.3]{KPQ97} since $\int_0^TE[|D_sh(X_T)|^2]\,ds<\infty.$
\end{proof}

  \subsubsection{Weighted Sobolev differentiable flow}
We now investigate differentiability properties of the solution with respect to the initial variable of the forward process.
  Let $0\le s\le t\le T$ and $x \in \mathbb{R}^d$.
  We denote by $(X^{s,x}, Y^{s,x}, Z^{s,x})$ the solution of the FBSDE
  \begin{equation}
\label{eq:fbsde sob}
 	\begin{cases}
		X_t = x + \int_s^tb(u,X_u, Y_u, Z_u)\,\diffns u + \int_0^t\sigma dW_{u}\\
		Y_t = h(X_T) + \int_t^Tg(u,X_u,Y_u, Z_u)\diffns u - \int_t^TZ_u\,dW_{u}\quad t\in [s,T]
	\end{cases}
\end{equation}
given by Theorem \ref{thm:exists fbsde}.
The next result gives regularity of the function $x\mapsto (X^{s,x}, Y^{s,x} )$.
Some notation need to be introduced before we state the result.
Let $\rho$ be a weight function.
That is, a measurable function $\rho :\mathbb{R}^d\to [0,\infty)$ satisfying
\begin{equation*}
	\int_{\mathbb{R}^d}(1 + |x|^p)\rho(x)\,dx<\infty
\end{equation*}
for some $p>1$.
Let $L^p(\mathbb{R}^d,\rho)$ be the weighted Lebesgue space of (classes) of measurable functions $f:\mathbb{R}^d\to \mathbb{R}^d$ such that
\begin{equation*}
	\|f\|_{L^p(\mathbb{R}^d,\rho)}^p:=\int_{\mathbb{R}^d}|f(x)|^p\rho(x)\,dx<\infty.
\end{equation*}
For functions $f:\mathbb{R}^d\to \mathbb{R}^l$ satisfying this integrability property we analogously define the space $L^p(\mathbb{R}^l, \rho)$.
Further denote by ${\cal W}^{1}_p(\mathbb{R}^d,\rho)$ the weighted Sobolev space of functions $f \in L^p(\mathbb{R}^d,\rho)$ admitting weak derivatives of first order $\partial_{x_i}f$ and such that
\begin{equation*}
	\|f\|_{{\cal W}^1_p(\mathbb{R}^d, \rho)} := \|f\|_{L^p(\mathbb{R}^d,\rho)} + \sum_{i=1}^d\|\partial_{x_i}f\|_{L^p(\mathbb{R}^d,\rho)}<\infty.
\end{equation*}
\begin{proposition}
\label{pro:sobolev flow}
	Assume that the conditions \ref{a1}-\ref{a4} are satisfied.
	\begin{itemize}
		\item[(i)] If condition \ref{b1} holds, then we have
	\begin{equation}
	\label{eq:x sobolev}
	X^{s,x}_t \in L^2\big(\Omega; {\cal W}^1_p(\mathbb{R}^d,\rho)\big)\quad \text{for every } t \in [0,T]
	\end{equation}
	and if $l=1$, then for every bounded open set $U \subseteq \mathbb{R}^d$ we have
	\begin{equation}
	\label{eq:y sobolev}
		Y^{s,x}_t \in L^2\big(\Omega; {\cal W}^1_1(U)\big) \quad \text{for every } t \in [0,T - \delta ]\quad \text{and every } \delta >0.
	\end{equation}	
	\item[(ii)] If condition \ref{b2} holds and $l=1$, then \eqref{eq:x sobolev} and \eqref{eq:y sobolev} hold with $\delta = 0$.
	\end{itemize}
\end{proposition}

\begin{proof}
	Recall from Theorem \ref{thm:exists fbsde} that the solution $(X,Y,Z)$ of the FBSDE \eqref{eq:fbsde} satisfies $Y_t^{s,x} = v(t, X^{s,x}_s)$ and $Z^{s,x}_t = w(t, X_t^{s,x})\sigma$ for some bounded measurable function $v:[0,T]\times\mathbb{R}^d\to \mathbb{R}^l$ and a measurable function $w:[0,T]\times \mathbb{R}^d\to \mathbb{R}^{l\times d}$.
	Thus, $X^{s,x}$ satisfies
	\begin{equation*}
		X^{s,x}_t = x + \int_s^tb(u, X^{s,x}_u, v(u, X^{s,x}_u), w(u,X^{s,x}_u)\sigma)\,du + \sigma W_t.
	\end{equation*}
	Under both conditions \ref{b1} and \ref{b2} the function $x\mapsto b(t, x, v(t,x),w(t,x)\sigma)$ is bounded and measurable.
	Thus, it follows from \cite[Theorem 3]{MNP2015} that $X^{s,x}_t\in L^2\big(\Omega; {\cal W}^1_p(\mathbb{R}^d,\rho)\big).$
	
	To deduce differentiability of $Y$, recall that for every $\delta >0$ and every $t \in [0,T - \delta]$ the function $x\mapsto v(t, x)$ is Lipschitz continuous.
	Let $\rho$ be the weight function given by $\rho(x):= 1_{U}(x)$.
	There is a measurable $N\subseteq\Omega$ such that $X^{s,\cdot}_t(\omega)\in {\cal W}^1_p(U)$ for all $\omega \in N^c$ and $P(N)=0$.
	Thus, by the chain rule formula, of \cite[Theorem 1.1]{Leo-Mor05}, for every $\omega \in N^c$ the function $Y^{s,x}_t(\omega) = v(t, X^{s,x}_t(\omega))$ belongs to the Sobolev space ${\cal W}^1_1(U)$.

	When condition \ref{b2} is satisfied, the function $x \mapsto v(t,x)$ is Lipschitz continuous for every $t\in [0,T]$.
	The claim (ii) then follows from the same arguments as above.
\end{proof}

\begin{appendix}
\section{A priori estimations for quasi-linear PDEs}
For the reader's convenience, in this appendix we collect some a priori estimations for quasi-linear PDEs.
These are fundamental for the proofs of our main results.
Different versions of these estimates can be found e.g. in \cite{Delarue03,Ma-Zhang11,Liberman} or \cite{Ladyz-Sol-Ura68}.
The results we present here are taken from \cite{Delarue03,Ma-Zhang11}.

Recall that the Sobolev space ${\cal W}^{1,2}_{p, \mathrm{loc}}((0,T)\times \mathbb{R}^d,\mathbb{R}^l)$ is the space of all functions
$u:(0,T)\times \mathbb{R}^d\to \mathbb{R}^l$ such that for all $r>0$,
\begin{equation*}
	\int_{(0,T)\times B_r(0)}\Big(|u|^{p} + |\partial_tu|^{p} + |D_x u|^{p} + |D_{xx} u|^{p} \Big)\,dx\,dt<\infty
\end{equation*}
and consider the quasilinear parabolic PDE
\begin{equation}
\label{eq:pde}
	\begin{cases}
		\partial_tv(t,x) + {\cal L}v(t,x) + g(t,x, v(t,x),D_x v(t,x)\sigma) = 0\\
		v(T,x) = h(x)
	\end{cases}
\end{equation}
where ${\cal L}$ is the second order differential operator
\begin{equation*}
  	{\cal L}v := b(t,x, v, D_x v\sigma)D_x v + \frac 12\text{trace}(\sigma\sigma^*D_{xx} v).
\end{equation*}

\begin{theorem}(\cite[Theorem 3.1 $\&$ Lemma 6.2]{Ma-Zhang11})
\label{thm:pde estimate}
	Assume that the conditions \ref{a1}-\ref{a4} are satisfied, and further assume that the functions $b$, $g$ and $h$ are bounded, smooth and with bounded derivatives.
	Let $v$ be the unique classical solution of \eqref{eq:pde}.
	Then for any $\delta>0$ there is $\alpha \in (0,1)$ and constants $C, C_\delta$ and $C_{\delta, \alpha}$ depending on $k_1, k_2, k_3,\Lambda, T,l$, $m$, and the bound of $b,g$ and which do not depend on the derivatives of $b,g$ such that
	\begin{itemize}
		\item[(i)] $|D_x v(t,x) |\le C_\delta$ for all $(t,x) \in [0,T - \delta]\times \mathbb{R}^d$
		\item[(ii)] for all $(t,x),(t',x')\in [0,T - \delta]\times \mathbb{R}^d$, it holds that
		\begin{equation*}
			|D_x v(t,x) - D_x v(t',x')|\le C_{\delta, \alpha}(|x-x'|^{\alpha} + |t-t'|^{\alpha/2}).
		\end{equation*}
		\item[(iii)] for every bounded domain $\mathcal{O} \subseteq \mathbb{R}^d$ and $p\ge 2$ it holds
		\begin{equation*}
			\int_0^{T-\delta}\int_{\mathcal{O}}\Big[ |D_x v(t,x)|^p + |D_{xx} v(t,x)|^p \Big]\,dx\,dt \le C_\delta^p|\mathcal{O}|
		\end{equation*}
		where $|\mathcal{O}|$ is the Lebesgue measure of $\mathcal{O}$.
	\end{itemize}
	If $h$ is twice continuously differentiable with bounded first and second derivatives, then (i), (ii) and (iii) hold with $\delta =0$ and $C_0$ may depend on $\|D_x h\|_\infty$ and $\|D_{xx} h\|_\infty$ as well.
\end{theorem}

\begin{theorem}(\cite[Theorems 1.3 $\&$ 2.9]{Delarue03})
\label{thm:pde estimate Delarue}
	Assume that the conditions \ref{a1}-\ref{a4} are satisfied and that $h$ is $\alpha$-H\"older continuous.
	Let $v$ be a solution of \eqref{eq:pde} in the space ${\cal W}^{1,2}_{d+1, \mathrm{loc}}((0,T)\times \mathbb{R}^d,\mathbb{R}^l)$.
	Then there are constants $C>0$ and $\alpha'\in (0, \alpha]$ depending only on $k_1,k_2, k_3, \Lambda, T,l$ and $m$ such that
	\begin{equation*}
		|v(t,x) - v(t',x')| \le C(|x - x'|^{\alpha'} + |t-t'|^{\alpha'/2})
	\end{equation*}
	for every $(t,x), (t',x')\in [0,T]\times \mathbb{R}^d$.
	If $\alpha = 1$, then it holds that
	\begin{equation*}
		|D_x v(t,x)| \le C \quad \text{for every } (t,x) \in [0,T]\times \mathbb{R}^d.
	\end{equation*}
\end{theorem}
  \end{appendix}


\vspace{.2cm}

	\noindent Peng Luo: Department of Statistics and Actuarial Science, University of Waterloo,
Waterloo,  ON,  N2L 3G1; Canada
 {\small\textit{E-mail address:} peng.luo@uwaterloo.ca}

  \vspace{.2cm}

	\noindent Olivier Menoukeu-Pamen: University of Liverpool Institute for Financial and Actuarial Mathematics, Department of Mathematical Sciences,
L69 7ZL, United Kingdom and African Institute for Mathematical Sciences, Ghana.
\small{\textit{E-mail address:} menoukeu@liverpool.ac.uk\\
Financial support from the Alexander von
Humboldt Foundation, under the programme financed by the German Federal Ministry of Education and Research
entitled German Research Chair No 01DG15010 is gratefully acknowledged.
  \vspace{.2cm}

	\noindent Ludovic Tangpi: Department of Operations Research and Financial Engineering, Princeton University, Princeton, 08540,
 NJ; USA.
 {\small\textit{E-mail address:} ludovic.tangpi@princeton.edu}\\

     \end{document}